\documentclass[12pt]{amsart}

\usepackage[matrix,arrow,curve,frame]{xy}
\usepackage{amsmath,amsthm,amssymb,enumerate}
\usepackage{latexsym}
\usepackage{amscd}
\usepackage[colorlinks=false]{hyperref}
\usepackage{euscript}

\newtheorem{thm}[subsection]{Theorem}
\newtheorem{defn}[subsection]{Definition}
\newtheorem{prop}[subsection]{Proposition}

\newtheorem{cor}[subsection]{Corollary}
\newtheorem{lemma}[subsection]{Lemma}

\newtheorem{remark}[subsection]{Remark}
\newtheorem{remarks}[subsection]{Remarks}

\theoremstyle{definition}
\newtheorem{example}[subsection]{Example}

\numberwithin{equation}{section}

\def\quot#1#2{#1/\!\!/#2}

\def\C{\mathbb {C}}

\def\NN{\mathcal N}
\def\OO{\mathcal O}

\def\N{\mathbb N}

\def\SL{\operatorname{SL}}
\def\GL{\operatorname{GL}}

\def\SO{\operatorname{SO}}
\def\Orth{\operatorname{O}}

\def\inv{^{-1}}
\def\lie#1{{\mathfrak #1}}
\def\lieg{\lie g}

\def\phi{{\varphi}}

\def\O{\mathcal O}

\def\pr{{\operatorname{pr}}}

\def\ra{\rightarrow}
\def\bra{\langle}
\def\ket{\rangle}

\def\cA{{\mathcal A}}

\def\cO{{\mathcal O}}

\def\cV{{\mathcal V}}

\def\gg{{\mathfrak g}}

\def\gl{{\mathfrak l}}

\def\go{{\mathfrak o}}
\def\gp{{\mathfrak p}}

\def\gs{{\mathfrak s}}

\def\codim{\operatorname{codim}}
\def\ZZ{\mathbb Z}

\DeclareMathOperator{\Hom}{Hom}

\DeclareMathOperator{\Sp}{Sp}

\DeclareMathOperator{\Spec}{Spec}

\DeclareMathOperator{\wt}{{wt}}

\newfont{\german}{eufm10}

\begin{document}
\pagestyle{plain}

\title
{Jet schemes and invariant theory}

\author{Andrew R. Linshaw, Gerald W. Schwarz, and Bailin Song}
\address{Department of Mathematics, University of Denver}
\email{andrew.linshaw@du.edu}

\address{Department of Mathematics, Brandeis University}
\email{schwarz@brandeis.edu}

\address{Department of Mathematics, University of Science and Technology of China}
\email{bailinso@ustc.edu.cn}
\thanks{The authors thank Peter Littelmann for useful conversations.}

\begin{abstract} Let $G$ be a complex reductive group and  $V$ a $G$-module. Then the $m$th jet scheme $G_m$ acts on the $m$th jet scheme $V_m$ for all $m\geq 0$. We are interested in the invariant ring $\cO(V_m)^{G_m}$ and whether the map $p_m^*\colon\cO((\quot VG)_m) \rightarrow \cO(V_m)^{G_m}$ induced by the categorical quotient map $p\colon V\ra \quot VG$ is an isomorphism, surjective, or neither. Using Luna's slice theorem, we give criteria for $p_m^*$ to be an isomorphism for all $m$, and we prove this when $G=\SL_n$, $\GL_n$, $\SO_n$, or $\Sp_{2n}$ and $V$ is a sum of copies of the standard module and its dual, such that $\quot VG$ is smooth or a complete intersection. We classify all representations of $\mathbb{C}^*$ for which $p^*_{\infty}$ is surjective or an isomorphism. Finally, we give examples where $p^*_m$ is surjective for $m=\infty$ but not for finite $m$, and where it is surjective but not injective. \end{abstract}

\maketitle

\section{Introduction}
Given an irreducible scheme $X$ of finite type over an algebraically closed field $k$, the first jet scheme $X_1$ is just the total tangent space of $X$. For $m>1$, the $m^{\text{th}}$ jet scheme $X_m$ is a higher-order generalization that is determined by its functor of points. For every $k$-algebra $A$, we have a bijection
$$\text{Hom} (\text{Spec}  (A), X_m) \cong \text{Hom} (\text{Spec}  (A[t]/\langle t^{m+1}\rangle ), X).$$ When $X$ is nonsingular, $X_m$ is irreducible for all $m\geq 1$, and is an affine bundle over $X$ with fiber an affine space of dimension $m\ \text{dim}(X)$. If $X$ is singular, the jet schemes are much more subtle and carry information about the singularities of $X$. The structural properties of $X_m$ are of interest, in particular the question of when $X_m$ is irreducible for all $m$. Mustata has shown that this holds when $X$ is locally a complete intersection with rational singularities, although these are not necessary conditions \cite{Mu}.

There are projections $X_{m+1} \rightarrow X_{m}$, and the {\it arc space} is defined to be $$X_{\infty}=\lim_{\leftarrow} X_m.$$ Even though it is generally not of finite type, $X_{\infty}$ has some nicer properties than $X_m$; for example, it is always irreducible \cite{Kol}. Arc spaces were originally studied by Nash in an influential paper \cite{Na}, in which he asked whether there is a bijection between the irreducible components of $X_{\infty}$ lying over the singular locus of $X$, and the essential divisors over $X$. This question is known as the {\it Nash problem}. It has been answered affirmatively for many classes of singular varieties, although counterexamples are known \cite{IK}. Arc spaces are also important in Kontsevich's theory of {\it motivic integration} \cite{Kon}. Given a complex algebraic variety $X$ and a resolution of singularities $Y \rightarrow X$ such that the discrepancy divisor $D$ has simple normal crossings, the motivic integral of $X$ is the integral of a certain function $F_D$ defined on the arc space $Y_{\infty}$, with respect to a measure on $Y_{\infty}$. Unlike ordinary integration, this measure takes values not in $\mathbb{R}$, but in a certain completion of the Grothendieck ring of algebraic varieties. Motivic integration was originally used by Kontsevich to prove that birationally equivalent Calabi-Yau manifolds have the same Hodge numbers. This theory was subsequently developed by many authors including Batyrev, Denef, Loeser, Looijenga, Craw, and Veys \cite{Bat}\cite{DLI}\cite{Loo}\cite{Cr}\cite{Ve}. A survey of these ideas and some of their applications can be found in \cite{DLII}.

Our goal in this paper is to establish some foundational results on the interaction between jet schemes, arc spaces, and {\it classical invariant theory}. If $G$ is a complex reductive group, $G_m$ is an algebraic group which is a unipotent extension of $G$. Let $Y$ be an affine $G$-variety and let $p\colon Y\to\quot YG=\text{Spec}(\cO(Y)^G)$ be the categorical quotient. Then $p$ induces a morphism $p_m\colon Y_m\ra(\quot YG)_m$ and a homomorphism
\begin{equation} \label{natmap}
p_m^*:\cO((\quot YG)_m) \rightarrow \cO(Y_m)^{G_m},
\end{equation} 
which was studied in some special cases by Eck in \cite{E} and by Frenkel-Eisenbud in the appendix of \cite{Mu}. We will find criteria for when this map is an isomorphism, surjective, or neither. First, using Luna's slice theorem, we show that all of these are local conditions (see Corollary \ref{cor:D-finite}). We are most interested in the case where $Y$ is a $G$-module $V$. Under mild hypotheses (see Corollary \ref{cor:coregii}), we show that when $\quot VG$ is smooth, $p_m^*$ is an isomorphism for all $m$. In Section 4, we give a more refined criterion for $p_m^*$ to be an isomorphism for all $m$ (see Theorem \ref{thm:ci}) and we show that it holds when $G=\SL_n$, $\GL_n$, $\SO_n$, or $\Sp_{2n}$ and $V$ is a sum of copies of the standard representation and its dual, such that $\quot VG$ is a complete intersection. In Section 5, we consider representations of $\mathbb{C}^*$. Using techniques of standard monomial theory, we classify all cases where $p^*_{\infty}$ is surjective, and we show that $p_{\infty}^*$ is an isomorphism whenever it is surjective. In Section 6, we show that for $G=\SL_n$ and $V=\ell\C^n$, $p^*_{\infty}$ is surjective, even though $p^*_m$ generally fails to be surjective for finite values of $m$. For $n=2$, $p^*_{\infty}$ is injective, but it is not injective for $n\geq 3$. The question of whether $p^*_{\infty}$ is surjective for arbitrary representations $V= k \C^n \oplus \ell (\C^n)^*$ of $\SL_n$, and similar questions for the other classical groups, remain open. 

Note that $\cO(V_\infty)^{G_\infty}$ is finitely generated as a differential algebra whenever $p_{\infty}^*$ is surjective, since $\cO((\quot VG)_\infty)$ is generated by $\cO(\quot VG)$ as a differential algebra. An interesting problem is to find sufficient conditions for $\cO(V_\infty)^{G_\infty}$ to be finitely generated as a differential algebra even if $p_{\infty}^*$ is not surjective. There are currently no examples where this is known to occur. Computer experiments suggest that this is the case for $G = \mathbb{C}^*$ and $V=\C^2$ with weights $2$ and $-3$ (see Example \ref{ex:torus}).

Our results have a number of applications to the theory of {\it vertex algebras} that appear in separate papers. Vertex algebras are a class of nonassociative, noncommutative algebras that arose out of conformal field theory in the 1980s, and in the work of Borcherds \cite{Bor} on the Moonshine conjecture. They were developed mathematically from several different points of view in the literature \cite{BD}\cite{FBZ}\cite{FLM}\cite{Ka}. An {\it abelian} vertex algebra is just a commutative ring equipped with a derivation. For any variety $X$, the ring $\cO(X_{\infty})$ has a derivation $D$ which makes it an abelian vertex algebra. On the other hand, many nonabelian vertex algebras $\cA$ possess filtrations for which the associated graded algebra $\text{gr}(\cA)$ is abelian and can be interpreted as $\cO(X_{\infty})$ for some $X$. 

The first application of our results is to the {\it commutant problem}. Given a vertex algebra $\cV$ and a subalgebra $\cA\subset \cV$, the commutant $\text{Com}(\cA,\cV)$ is the subalgebra of $\cV$ that commutes with $\cA$. In \cite{LSS}, interesting examples of commutants were described using the fact that $\text{gr}(\text{Com}(\cA,\cV))$ is isomorphic to $\cO((\quot VG)_\infty)$ for a certain choice of $V$ and $G$. This leads to vertex algebra analogues of the classical Howe pairs of types $\GL_n - \GL_m$, $\SO_n - \gs\gp_{2m}$, and $\Sp_{2n} - \gs\go_{2m}$. The second application of our results is to the {\it chiral de Rham complex} \cite{MSV}. This is a sheaf of vertex algebras on any nonsingular variety or complex manifold $X$ that contains the ordinary de Rham sheaf at weight zero, and captures stringy invariants of $X$ such as the elliptic genus. Using the fact that $p^*_{\infty}$ is an isomorphism for $G = \SL_2$ and $V = \ell \mathbb{C}^2$, Song gave a complete description of the global section algebra when $X$ is a Kummer surface; it is isomorphic to the $N=4$ superconformal algebra with $c=6$ \cite{So}. Previously, the only nontrivial case where a description was known was $\mathbb{C}\mathbb{P}^n$ \cite{MS}.

\section{Jet schemes}
Throughout this paper our base field will be $\mathbb{C}$. We recall some basic facts about jet schemes, following the notation in \cite{EM}. Let $X$ be an irreducible scheme of finite type. For each integer $m\geq 0$, the jet scheme $X_m$ is determined by its functor of points: for every $\mathbb{C}$-algebra $A$, we have a bijection
$$\text{Hom} (\text{Spec}  (A), X_m) \cong \text{Hom} (\text{Spec}  (A[t]/\langle t^{m+1}\rangle ), X).$$ Thus the $\mathbb{C}$-valued points of $X_m$ correspond to the $\mathbb{C}[t]/\langle t^{m+1}\rangle$-valued points of $X$. If $p>m$, we have projections $\pi_{p,m}: X_p \rightarrow X_m$ and $\pi_{p,m} \circ \pi_{q,p} = \pi_{q,m}$  when $q>p>m$. Clearly $X_0 = X$ and $X_1$ is the total tangent space $\text{Spec}(\text{Sym}(\Omega_{X/\mathbb{C}}))$. The assignment $X\mapsto X_m$ is functorial, and a morphism $f:X\ra Y$ induces $f_m: X_m \ra Y_m$ for all $m\geq 1$. If $X$ is nonsingular, $X_m$ is irreducible and nonsingular for all $m$. Moreover, if $X,Y$ are nonsingular and $f:X\ra Y$ is a smooth surjection, $f_m$ is surjective for all $m$. 

If $X=\text{Spec}(R)$ where $R= \mathbb{C}[y_1,\dots,y_r] / \langle f_1,\dots, f_k\rangle$, we can find explicit equations for $X_m$. Define new variables $y_j^{(i)}$ for $i=0,\dots, m$, and define a derivation $D$ by $D(y_j^{(i)}) = y_j^{(i+1)}$ for $i<m$, and $D(y_j^{(m)}) =0$, which specifies its action on all of $\mathbb{C}[y_1^{(0)},\dots, y_r^{(m)}]$. In particular, $f_{\ell}^{(i)} = D^i ( f_{\ell})$ is a well-defined polynomial in $\mathbb{C}[y_1^{(0)},\dots, y_r^{(m)}]$. Letting $R_m = \mathbb{C}[y_1^{(0)},\dots, y_r^{(m)}] / \langle f_1^{(0)},\dots, f_k^{(m)}\rangle$, we have $X_m\cong \text{Spec}(R_m)$. By identifying $y_j$ with $y_j^{(0)}$, we see that $R$ is naturally a subalgebra of $R_m$. There is a $\mathbb{Z}_{\geq 0}$-grading $R_m = \bigoplus_{n\geq 0} R_m[n]$ by weight, defined by $\wt(y^{(i)}_j) = i$. For all $m$, $R_m[0] = R$ and $R_m[n]$ is an $R$-module.

Given a scheme $X$, define $$X_{\infty}  = \lim_{\leftarrow } X_m,$$ which is known as the {\it arc space} of $X$. For a $\mathbb{C}$-algebra $A$, we have a bijection
$\text{Hom} (\text{Spec} (A), X_{\infty} ) \cong \text{Hom} (\text{Spec} A[[t]], X)$. We denote by $\psi_m$ the natural map $X_{\infty} \rightarrow X_m$. If $X = \text{Spec}(R)$ as above, $$X_{\infty} \cong \text{Spec} (R_{\infty}), \text{where} \ R_{\infty}  = \mathbb{C}[y_1^{(0)},\dots,y_j^{(i)},\dots] / \bra f_1^{(0)},\dots, f_\ell^{(i)},\dots\ket.$$ Here $i=0,1,2,\dots$ and $D (y^{(i)}_j) = y^{(i+1)}_j$ for all $i$. By a theorem of Kolchin \cite{Kol}, $X_{\infty}$ is irreducible whenever $X$ is irreducible.

\section{Group actions on jet schemes}\label{sec:groupsacting}
We establish some elementary properties of jet schemes and quotient mappings for reductive group actions. Mainly we see what one can say using Luna's slice theorem \cite{LunaSlice}. Let $G$ be a complex reductive algebraic group with Lie algebra $\gg$. For $m\geq 1$, $G_m$ is an algebraic group which is the semidirect product of $G$ with a unipotent group $U_m$. The Lie algebra of $G_m$ is $\gg[t]/t^{m+1}$. Given an affine $G$-variety $Y$, there is the quotient $Z:=\quot YG=\Spec(\OO(Y)^G)$ and the canonical map $p\colon   Y\to Z$ (sometimes denoted $p_Y$) which is dual to the inclusion $\OO(Y)^G\subset\OO(Y)$.  We have 
a natural action of $G_m$ on $Y_m$, and we are interested in the invariant ring $\OO(Y_m)^{G_m}$, the morphism $p_m \colon  Y_m\to Z_m$ and whether $p_m^*: \OO(Z_m)\ra \OO(Y_m)^{G_m}$ is an isomorphism, surjective, or neither. If $\phi\colon X\to Y$ is a morphism of affine $G$-varieties, then $\quot \phi G$ will denote the induced mapping of $\quot XG$ to $\quot YG$.

Recall that a morphism of varieties is {\it \'etale} if it is smooth with fibers of dimension zero. If $\phi\colon X\to Y$ is a morphism where $X$ and $Y$ are smooth, then $\phi$ is \'etale if and only if $d\phi_x: T_xX\to T_{\phi(x)}Y$ is an isomorphism for all $x\in X$.

\begin{defn}
Let $G$ be a reductive complex algebraic group and let $\phi\colon X\to Y$ be an equivariant map of affine $G$-varieties. We say that $\phi$ is \emph{excellent\/} if the following hold.
\begin{enumerate}
\item $\phi$ is \'etale.
\item $\quot\phi G:\quot XG\to \quot YG$ is \'etale.
\item The canonical map $(\phi,p_X)\colon X\to Y\times_{\quot YG}\quot XG$ is an isomorphism.
\end{enumerate}
\end{defn}

Note that condition (1) is a consequence of conditions (2) and (3). Let us say that $X$ is \emph{$m$-very good \/} if $p_m^*: \O((\quot XG)_m) \ra \O(X_m)^{G_m}$ is an isomorphism. We say that $X$ is \emph{$m$-good \/} if $p_m^*$ is surjective, so that $\O(X_m)^{G_m} = p_m^*\O((\quot XG)_m)$, and we say that $X$ is \emph{$m$-bad\/} if $p_m^*$ is not surjective. Here $m$ is finite or $\infty$. Usually we drop the $m$. We say that $X$ is \emph{$D$-finite\/} if $\cO(X_\infty)^{G_\infty}$ is finitely generated as a differential algebra.

\begin{lemma}\label{lem:excellent}
Suppose that $\phi\colon X\to Y$ is excellent.   Then
\begin{enumerate}
\item $X_m\simeq \quot XG\times_{\quot YG} Y_m$.
\item $(\quot XG)_m\simeq \quot XG\times_{\quot YG}(\quot YG)_m$.
\end{enumerate}
If  $Y$ is very good (resp.\ good or $D$-finite), then so is $X$, and conversely if $\phi$ is surjective.
\end{lemma}

\begin{proof}
Since $\phi$ is \'etale,  $X_m\simeq X\times_Y Y_m$. Since $\phi$ is excellent,  
$$
X\times_Y Y_m\simeq \quot XG\times_{\quot YG} Y\times_Y Y_m\simeq \quot XG\times_{\quot YG} Y_m
$$ 
and since $\quot \phi G$ is \'etale, $(\quot XG)_m\simeq \quot XG\times_{\quot YG}(\quot YG)_m$. Thus we have (1) and (2).

If $Y$ is very good, then by (1) and (2)
$$
\O(X_m)^{G_m}\simeq   \O(\quot XG)\otimes_{\O(\quot YG)}\O((\quot YG)_m)\simeq \O((\quot XG)_m),
$$
hence $X$ is very good. Similarly, $Y$ good implies that $X$ is good.

Conversely, if $\phi$ is surjective and $\O(Y_m)^{G_m}\neq p_{Y,m}^*\O((\quot YG)_m)$, then, since $\quot\phi G$ is faithfully flat, we have that 
$$
\O(\quot XG)\otimes_{\O(\quot YG)} \O(Y_m)^{G_m}\neq \O(\quot XG)\otimes_{\O(\quot YG)}p_{Y,m}^*\O((\quot YG)_m)
$$
 and hence that $\O(X_m)^{G_m}\neq p_{X,m}^*\O((\quot XG)_m)$. Hence $Y$ is good if  $X$ is good. The proof that $Y$ is very good if $X$ is very good is similar.

Now 
$$
\cO(X_{\infty})^{G_{\infty}} \cong   \cO(Y_{\infty})^{G_{\infty}}\otimes_{\cO(Y)^G}\cO(X)^G.
$$
Thus if $Y$ is $D$-finite, then clearly so is $X$. Conversely, assume that $X$ is $D$-finite and that $\phi$ is surjective. Set  $A:=\cO(X_\infty)^{G_\infty}$. Then we have the weight grading $A=\oplus_{n\in\N} A_n$  where $A_0=\O(X)^G$.  Let $B$ denote $\cO(Y_\infty)^{G_\infty}$. Then $B$ is graded and the isomorphism $A\cong B\otimes_{\cO(Y)^G} \cO(X)^G$ is an isomorphism of graded rings.  
Let $f_i\otimes h_i$ be generators of $A\simeq B\otimes_{\cO(Y)^G}\cO(X)^G$ as differential graded algebras. We may assume that each $f_i$ has weight $n_i$ for some $n_i\in\N$.  Let $p_1,\dots,p_d$ be generators of $\cO(X)^G$. Then $Dp_j=\sum f_{ij}\otimes h_{ij}$ where the $h_{ij}$ are elements of $\cO(X)^G$ and the $f_{ij}$ are in $B_1$. Now take the collection of elements $f_i$ and $f_{ij}$ in $B$.  An induction argument shows that $D$ applied repeatedly to the elements $f_i\otimes h_i$ ends up in the $\cO(X)^G$-submodule of $A$ generated by $D$ applied to products of the elements $f_i$ and $f_{ij}$ . Since $\phi$ is faithfully flat, this shows that the $B_0$-submodule of $B_n$ generated by the elements $f_i$ and $f_{ij}$ is  $B_n$ since this submodule tensored with $\cO(X)^G$ is $A_n$. Hence the $f_i$ and $f_{ij}$ generate $B_n$ for all $n$ and $Y$ is $D$-finite. 
\end{proof}

A subset $S$ of $X$ is \emph{$G$-saturated\/}  if $S=p\inv(p(S))$; equivalently, $S$ is a union of fibers of $p$.

\begin{cor}\label{cor:D-finite}
\begin{enumerate}
\item Suppose that $X=\cup X_\alpha$ where the $X_\alpha$ are Zariski open and $G$-saturated. Then $X$ is very good  (resp.\  good or $D$-finite) if and only if each $X_\alpha$ is very good (resp.\ good or $D$-finite).
\item Let $W$ be a $G$-module and $U=W_f$ where $f\in\O(W)^G$ and $f(0)\neq 0$. Then $W$ is very good (resp.\  good or $D$-finite) if and only  $U$ is very good (resp.\ good or $D$-finite).
\end{enumerate}
\end{cor}

\begin{proof}
For (1) we may assume that we have a finite cover. Then the canonical map $\amalg X_\alpha\to X$ is excellent and surjective and (1) follows from Lemma \ref{lem:excellent}.
For (2) we may assume that $U$ is very good (resp.\  good or $D$-finite).   Now $W$ is the union of $U$ and finitely many translates $U_\lambda$ where $U_\lambda=\lambda\cdot U$ for $\lambda\in\C^*$. Clearly each $U_\lambda$ is very good (resp.\ good or $D$-finite) since $U$ is. Thus we can apply (1).
\end{proof}

Let $H$ be a reductive subgroup of $G$ and $Y$ an affine $H$-variety. Then $G\times^HY$ denotes the quotient of $G\times Y$ by the $H$-action sending $(g,y)$ to $(gh\inv,hy)$ for $(g,y)\in G\times Y$ and $h\in H$. We denote the orbit of $(g,y)$ by $[g,y]$. We  have an action of $G$ on the left on $G\times Y$ which commutes with the action of $H$ and induces a $G$-action on $G\times^H Y$.  Then $\quot{(G\times^HY)}G\simeq\quot YH$.
Note that $G\to G/H$ is a principal $H$-bundle, hence trivial over pullback via an \'etale surjective map to $G/H$.

\begin{lemma}\label{lem:slices}
Let $H$ be a reductive subgroup of $G$ and $Y$ an affine $H$-variety. Then 
$$
 \O((G\times^H Y)_m)^{G_m}\simeq\O(Y_m)^{H_m}.
$$ 
Hence $G\times^HY$ is very good (resp.\ good or $D$-finite) if and only if $Y$ is very good (resp.\ good or $D$-finite). 
\end{lemma}
 
\begin{proof}
 For a trivial principal $H$ bundle $U\times H$, we have 
 $(U\times H)_m=U_m\times H_m$ is a trivial $H_m$-bundle with quotient $U_m$. Thus   $(G\times^H Y)_m$ is the quotient of $G_m\times Y_m$ by the action of $H_m$ (it is a principal bundle). 

Consider the action of $G_m$ on $G_m\times Y_m$. Then the quotient is clearly just projection to $Y_m$, so that $\O(G_m\times Y_m)^{G_m}\simeq\O(Y_m)$. Thus 
$$
\O(G_m\times Y_m)^{G_m\times H_m}\simeq\O(Y_m)^{H_m}
$$ 
so that $\O((G\times^H Y)_m)^{G_m}\simeq\O(Y_m)^{H_m}$.  
\end{proof}

Let $X$ be a smooth  affine $G$-variety and suppose that $Gx$ is a closed orbit. Then the isotropy group $H:=G_x$ is reductive and we have a splitting of $H$-modules  $T_xX=T_x(Gx)\oplus N$. The representation $(N,H)$ is called \emph{the slice representation at $x$\/}.   Here is Luna's slice theorem \cite{LunaSlice} in our context.

\begin{thm}\label{thm:slice}
\begin{enumerate}
\item There is a locally closed
affine $H$-stable and  $H$-saturated subvariety $S$ of $X$ containing
$x$ such that $U := G\cdot S$ is a  $G$-saturated affine open subset of
$X$. Moreover, the canonical $G$-morphism 
$$
\phi \colon G\times^HS\to U,\qquad [g,s] \mapsto gs
$$
is excellent.
\item $S$ is smooth at $x$ and the $H$-modules $T_xS$ and $N$ are isomorphic. 
Possibly shrinking $S$ we can arrange that
there is an excellent surjective $H$-morphism  $\psi : S
\to N_f$ which sends $x$  to $0$,  inducing an excellent
$G$-morphism   
$$
\tau :G\times^HS\to G\times^H N_f
$$
where $f\in\O(N)^H$ and $f(0)\neq 0$.
\end{enumerate}
\end{thm}

Combining \ref{lem:excellent}--\ref{thm:slice} we obtain

\begin{cor}\label{cor:slicereps}
Suppose that  $X$  is smooth. Let $(W,H)$ be a slice representation of $X$.
 \begin{enumerate}
\item If $X$ is very good (resp.\ good or $D$-finite) then so is $W$.
\item If $W$ is very good (resp.\ good or $D$-finite)  for each slice representation $(W,H)$ of $X$, then $X$ is very good (resp.\ good or $D$-finite).
\end{enumerate}
\end{cor}
 
From Lemma \ref{lem:excellent} we obtain 
 
\begin{cor}\label{cor:localization}
Suppose that $X$ is smooth and that $f\in\cO(X)^G$.
\begin{enumerate}
\item If the slice representations of $X_f$ are very good, then 
$$
(\cO(X_m)^{G_m})_f\simeq \cO((\quot XG)_m)_f.
$$
\item If the slice representations of $X_f$ are good, then 
$$
(\cO(X_m)^{G_m})_f\simeq (p_m^*\cO((\quot XG)_m))_f.
$$
\end{enumerate}
\end{cor}

Our main focus in this paper will be on the case where $X$ is a finite-dimensional $G$-module $V$. Choose a basis $\{x_1,\dots,x_n\}$ for $V^*$, so that $$\cO(V) \cong  \mathbb{C}[x_1,\dots,x_n], \qquad \cO(V_m) =  \mathbb{C}[x_1^{(i)},\dots,x_n^{(i)}], \qquad 0\leq i\leq m.$$ The action of $G_m$ on $V_m$ induces the following action of its Lie algebra $\gg[t]/t^{m+1}$ on $\cO(V_m)$. For $\xi\in \gg$, \begin{equation*}\label{jetaction}\xi t^r (x_j^{(i)}) = \lambda^r_i (\xi(x_j))^{(i-r)}, \qquad \lambda^r_i = \bigg\{ \begin{matrix} \frac{i!}{(i-r)!}  & 0\leq r\leq i \cr 0 & r>i \end{matrix}.\end{equation*} The invariant ring  $\cO(V_m)^{\gg[t]/t^{m+1}}$ coincides with $\cO(V_m)^{G_m}$ when $G$ is connected.

 \begin{lemma} \label{directsum}
 Suppose that $V\oplus W$ is a representation of $G$. If $W$ is bad (resp.\ not $D$-finite) then so is $V\oplus W$. 
 \end{lemma}
 
 \begin{proof}
 If $W$ is bad there is a $G_m$- invariant polynomial on $W_m$ which does not come from $\O((\quot WG)_m)$. Then clearly it cannot come from an element of $\O(\quot {((V\oplus W)}G)_m)$. Now minimal generators of $\cO(V_\infty\oplus W_\infty)^{G_\infty}$ can clearly be chosen to be bihomogeneous in the variables of $V_\infty$ and $W_\infty$. Thus if $V\oplus W$ is $D$-finite, then so is $W$.
 \end{proof}
 
The results above say  that a representation is bad (resp.\ not $D$-finite) if a subrepresentation or slice representation is bad (resp.\ not $D$-finite). Now let us consider some examples.

 \begin{example}\label{ex:pm1}
Let $(V,G)=(\C,\pm 1)$. Then $V$ is bad. Let $z$ be a coordinate function on $V$. Then $V_1$ has coordinates $z=z^{(0)}$ and $z^{(1)}$. The invariants of $G=G_1$ are generated by $z^2$, $zz^{(1)}$ and $(z^{(1)})^2$. The invariants coming from the quotient are $z^2$ and $2zz^{(1)}$. If one goes to degree $2$, then from $\C[z^2]$ we get $z^2$, $2zz^{(1)}$ and $2(z^{(1)})^2+2zz^{(2)}$. But among the $G_2$-invariants we have $z^2$, $zz^{(1)}$, $zz^{(2)}$, $(z^{(1)})^2$, $z^{(1)}z^{(2)}$ and $(z^{(2)})^2$. Things are only getting worse. See Theorem \ref{thm:badfinite} for the general case. \end{example}
  
 \begin{example} \label{ex:torus}Let $G = \C^*$ and let $V = \C^2$ with weights $2$ and $-3$. Then $\cO(V)^G$ is generated by $z=x^3y^2$, so that $\quot VG \cong  \C$. For $m=1$,
$w:=(D(z))^2 /z=x(3y  x^{(1)} + 2 x  y^{(1)})^2$ is not a function on $(\quot VG)_1$, but it is a $G_1$-invariant function on $V_1$. Hence $V$ is $1$-bad. In fact, it is $m$-bad for any $m\geq 1$.
See Theorem \ref{thm:dim1} for a generalization. Computer calculations suggest that $V$ is $D$-finite with generators $z$ and $w$. Thus this is likely an example where $V$ is bad yet $D$-finite.
 \end{example}

\begin{example} \label{sixcopies} 
Let $G = \SL_3$ and let $V$ be the direct sum of $6$ copies of the standard representation $\C^3$, with basis $\{x^{(a,0)}, y^{(a,0)}, z^{(a,0)}| \ a= 1,\dots,6\}$. The generators of $\cO(V)^G$ are $3\times 3$ determinants $[abc]$ corresponding to a choice of three distinct indices $a,b,c\in \{1,\dots, 6\}$. Let $x^{(a,1)} = D x^{(a,0)}$ and similarly for $y$ and $z$, and define 
$$
f = \sum_{\sigma\in \mathfrak{S}_6} \text{sgn}(\sigma) x^{(\sigma(1),0)} y^{(\sigma(2),0)}z^{(\sigma(3),0)} x^{(\sigma(4),1)} y^{(\sigma(5),1)} z^{(\sigma(6),1)} ,
$$ 
where $\sigma$ runs over the group $\mathfrak{S}_6$ of permutations of $\{1,\dots, 6\}$. Note that $f$ has degree $6$ and weight $3$ and lies in $\cO(V_1)^{G_1}$, but $f\notin p_1^*\cO((\quot VG)_1)$, since elements of $p_1^*\cO((\quot VG)_1)$ of degree $6$ can have weight at most $2$. Hence $V$ is $1$-bad, and in fact it is $m$-bad for all finite $m\geq 1$. However, $f$ can be expressed (up to a constant multiple) in the form 
$$
\sum_{\sigma\in \mathfrak{S}_6} [\sigma(1) \sigma(2) \sigma(3)] ([\sigma(4) \sigma(5) \sigma(6)])^{(3)},
$$ 
so $f\in p_3^*\cO((\quot VG)_3)$.  In Theorem \ref{thm:sln} and Example \ref{ex:sl3} below we show  that $p_{\infty}^*$ is surjective but not injective.
\end{example}

\begin{remark} The surjectivity of $p^*_{\infty}$ is equivalent to the condition that every element of $\cO(V_{\infty})^{G_{\infty}}$ of weight $m$ lies in $p_m^* \cO((\quot VG)_{m})$. \end{remark}

 \begin{thm}\label{thm:badfinite}
Suppose that $G\subset\GL(V)$ is finite and nontrivial. Then $V$ is $m$-bad for any $m\geq 1$ and $V$ is not $D$-finite.
\end{thm}

\begin{proof} Note that $G_m=G$ for all $m$. Let $k>0$ be minimal such that there is a homogeneous generator of $\O(V)^G$ of degree $k$.
Let $f_1,\dots,f_\ell$ be a basis of the generators of degree $k$. We have $V_1\simeq TV=V\oplus V'$ where $V'$ is a $G$-module isomorphic to $V$. Using the isomorphism we obtain   minimal generators $f_1',\dots,f_\ell'$ of $\O(V')^G$ which are linearly independent. The $f_i'$ exist in every $\O(V_m)^G$ for $m\geq 1$. They are not in $p_m^*\O((V/G)_m)$  because the only possibility is that $f_i'=D^kf_i$ for all $i$ where the latter have  terms involving the variables of $V_k$ not in $V_1$. Thus $V$ is $m$-bad for all $m\geq 1$.

Let $f$ be a homogeneous element of $\cO(V_\infty)^G$ of minimal positive degree, say $m$. Let $1\leq i_1<\dots<i_m\leq s$. Then there is a polarization  $f_{i_1,\dots,i_m}$ which is multilinear and invariant on the copies of $V$ in $V_s$ corresponding to the indices $i_1,\dots,i_m$. Now consider $i_j=rm+j$ for $r\geq 1$. If $f:=f_{i_1,\dots,i_m}$ is in the differential subalgebra of $\cO(V_\infty)^G$ generated by $\cO(V_{rm})^G$, then we must have that $f$ is a sum of $D$ to some powers applied to   invariants of degree $m$ lying in $\cO(V_{rm})^G$. But it is easy to see that such a sum can never give $f$. 
  \end{proof}

\begin{cor} If $X$ is a smooth affine $G$-variety and $(W,H)$ is a slice representation of $X$ with $H$ finite and acting nontrivially on $W$, then $X$ is bad and not $D$-finite.
\end{cor}

We need some more background on the action of $G$, see \cite{LunaSlice}. 
 The points of $\quot VG$ are in one to one correspondence with the closed orbits $Gv$, $v\in V$. Let $H:=G_v$ be the isotropy group (which is reductive) and let  $(W,H)$ be the slice representation.  The fiber $p\inv(p(v))$ is isomorphic to $G\times^H\NN(W)$ where $\NN(W):=p_W\inv(p_W(0))$ is the \emph{null cone\/} of $W$. For the rest of this section, we set $Z:= \quot VG$. Let $(H)$ denote the conjugacy class of $H$ in $G$ and let $Z_{(H)}$ denote the closed orbits $Gv'$ such that $G_{v'}\in (H)$. Then there are finitely many strata $Z_{(H)}$ each of which is smooth and irreducible. For reductive subgroups $H_1$ and $H_2$ of $G$, write $(H_1)\leq (H_2)$ if $H_1$ is $G$-conjugate to a subgroup of $H_2$. Then among the isotropy classes of closed orbits, there is a unique minimum $(H)$ with respect to $\leq$, called the \emph{principal isotropy class\/}. We also call $H$ a \emph{principal isotropy group} and corresponding closed orbits are called \emph{principal orbits}. Then  $Z_{(H)}$ is the unique open stratum in $Z$ and we also denote it by $Z_\pr$. Let $Gv$ be a principal orbit with $G_v=H$. Then the fiber of $p$ through $v$ is  of the form $G\times^HW$ where $W$ is the nontrivial part of the slice representation of $H$ at $v$ and $\OO(W)^H=\C$. We say that a $G$-module is {\it stable} if the general $G$-orbit is closed. Then the slice representations   of the principal isotropy groups are trivial and   $V_\pr:=p\inv(Z_\pr)$ is open and  consists of principal orbits.

Let $S$ be an irreducible hypersurface in $Z$. Then the ideal of $S$ is generated by an invariant $f$. Write $f=f_1^{a_1}\dots f_n^{a_n}$ where the $f_i$ are irreducible polynomials in $\OO(V)$. We say that the irreducible component $\{f_i=0\}$ of $p\inv(S)$ is \emph{schematically reduced\/} if $a_i=1$. Equivalently, the differential $df$ does not vanish at some point of $\{f_i=0\}$. We say that $p\inv(S)$ is schematically reduced if all of its irreducible components are. Equivalently, $f$ generates the ideal  of $p\inv(S)$ in $\OO(V)$.  The \emph{codimension one strata\/} of $V$ are the inverse images in $V$ of the codimension one strata of $Z$.

\begin{thm}\label{thm:dim1}
Let $V$ be a  $G$-module such that $\dim Z=1$. Then $Z\simeq\C$.
\begin{enumerate}
\item If $\O(V_m)^{G_m}=p_m^*\O(\C_m)$ for some $m\geq 1$, then an irreducible component of $\NN(V)$ is schematically reduced.
\item If  an irreducible component of $\NN(V)$ is schematically reduced, then $\O(V_m)^{G_m}=p_m^*\O(\C_m)$ for all $m\geq 1$.
\end{enumerate}
\end{thm}
  
\begin{proof}  Since $Z$ is normal of dimension one, we have $Z\simeq \C$ and $\OO(V)^G$ is generated by a homogeneous invariant $p$. Write $p=p_1^{a_1}\dots p_n^{a_n}$ where the $p_i$ are irreducible polynomials in $\OO(V)$. Suppose that no irreducible component of $\NN(V)$ is schematically reduced. Then $a_i\geq 2$ for all $i$ and $(Dp)^2$ is divisible by $p$, yet $(Dp)^2/p$ is not the pullback of an element of $\O(\C_m)$ and we have (1).

  Now suppose that an irreducible component of $\NN(V)$ is schematically reduced and that $V$ is stable.
  Let $V'=\{v\in V\mid dp(v)\neq 0\}$. Then $V'$ is $G$-stable, open and dense in $V$. Since $dp$ does not vanish somewhere on $\NN(V)$, $p$ is a smooth mapping of $V'$ onto $\C$.  Hence $p_m\colon   V'_m\to\C_m$ is smooth and surjective. The principal fibers of $p\colon V\to\C$ are homogeneous spaces $G/H$ where $H$ is reductive. Since $G\to G/H$ is a principal $H$-bundle, $G_m\to (G/H)_m$ is a principal $H_m$-bundle and $(G/H)_m\simeq G_m/H_m$.  It follows that the fibers of $p_m$ in $(V_\pr)_m$ are homogeneous spaces $G_m/H_m$. Hence any $h\in\O(V_m)^{G_m}$ is the pullback of a rational function $\tilde h$ on $\C_m$. If $\tilde h$ has poles, then so does $p_m^*\tilde h=h$. Hence $\tilde h$ is in $\O(\C_m)$ and we have proved (2) in case $V$ is stable.
  
    Now suppose that $V$ is not stable. Then the principal fibers are $G\times^HW$ where $\O(W)^H=\C$.  Since $\O(W)^H=\C$, by Hilbert-Mumford there  is a 1-parameter subgroup  $\lambda\colon \C^*\to H$ such that $HW_\lambda=W$ where $W_\lambda$ is the sum of the strictly  positive weight spaces of $\lambda$. There is a dense open subset $W'_\lambda$ of $W_\lambda$ such that $H\times W_\lambda'\to W'\subset W$ is surjective and smooth where $W'$ is open in $W$. Then $H_m\times W'_{\lambda,m}\to W'_m$ is surjective and smooth where $\lambda$ has only positive weights on $W'_{\lambda,m}:=(W'_\lambda)_m$. Hence the $H_m$-invariants of $W_m$ are just the constants. It follows that the $G_m$-invariants on $G_m\times^{H_m}W_m$ are constants and the proof above goes through.
\end{proof}

\begin{remark}  \label{rem:goodandverygood} Whenever $Z_m$ is irreducible and reduced, $p_m^*$ is injective since $p_m$ is dominant.  Hence good and very good are equivalent in this case.
\end{remark}
 
\begin{cor}[Eck]
Let $V$ be a stable $G$-module with $\dim Z=1$. Assume that the generating invariant $p$ is irreducible. Then $V$ is very good. \end{cor}

We say that the $G$-module $V$ is \emph{coregular\/} if $Z$ is smooth. Equivalently, $\cO(V)^G$ is a polynomial ring \cite[II.4.3 Lemma 1]{Kr}. In this case, good and and very good are equivalent by Remark \ref{rem:goodandverygood}.

 \begin{cor}\label{cor:coreg}
Let $V$ be coregular. Then $V$ is very good if and only if each codimension one stratum of $V$ has a schematically reduced irreducible component.
 \end{cor}
  
  \begin{proof} If a codimension one stratum has no schematically reduced irreducible component,  then the corresponding slice representation is of the form $(W+\theta,H)$ where $\theta$ is a trivial representation, $W^H=0$, $\dim \quot WH=1$ and $\NN(W)$ has no schematically reduced irreducible component. Then Corollary \ref{cor:slicereps} and Theorem  \ref{thm:dim1} show that $V$ is bad.  
  
Now assume that each codimension one stratum has a schematically reduced irreducible component.   Let $V'$ be the set of points of $V$ where $d p$ has maximal rank and let $Z'\subset Z$ be the image. Then the complement of $Z'$ has codimension at least 2 in $Z\simeq\C^k$. As in the case $k=1$, any $G_m$- invariant polynomial  on $V_m$ is the pullback of a rational function on $(\C^k)_m$ which has no poles on $Z'_m$. But the complement of $Z'_m$ in $(\C^k)_m$ has codimension at least 2. Hence our $G_m$-invariant polynomial is the pullback of a polynomial on $(\C^k)_m$.
\end{proof}

Now let $G$ be a connected complex reductive group and let $V$ be a $G$-module. We impose a mild technical condition which is automatic if $G$ is semisimple; we assume that $\cO(V)$ contains no nontrivial one-dimensional invariant subspaces. Equivalently, we assume that every semi-invariant of $G$ is invariant.

\begin{lemma}\label{lem:techcond}
Assume that $G$ is connected and that every semi-invariant of $\cO(V)$ is invariant. 
\begin{enumerate}
\item A function $f\in\O(V)^G$ is irreducible in $\O(V)$ if and only if it is irreducible in $\O(V)^G$. In particular, $\O(V)^G$ is a UFD.
\item The codimension one strata of $V$ are irreducible and schematically reduced.
\item Let $S\subset\quot VG$ have codimension at least $2$. Let $f_1$ and $f_2$ be relatively prime elements of $\O(V)^G$ which vanish on $S$. Then $f_1$ and $f_2$ are relatively prime elements of $\O(V)$, hence $p\inv(S)$ has codimension at least $2$ in $V$.
\end{enumerate}
\end{lemma}

\begin{proof} We only need to prove (1) since this implies both (2) and (3). Let $f\in \O(V)^G$ and let $f = p_1\cdots p_k$ be its prime factorization in $\O(V)$. Since every semi-invariant of $G$ is invariant, each $g\in G$ must permute the factors of $f$, so $f$ determines a homomorphism from $G$ to the permutation group on $k$ letters. But $G$ is connected and this map is continuous, so it must be trivial. Therefore each $p_i\in \O(V)^G$. \end{proof}

 \begin{cor}\label{cor:coregii}
Suppose that $V$ is coregular, $G$ is connected and   every semi-invariant of $\cO(V)$ is invariant. Then $V$ is very good. \end{cor}
 
As above, let $V'$ be the set of points in $V$ where $dp$ has maximal rank, and let $Z'\subset Z$ be the image of $V'$.  
 
 \begin{lemma}\label{lem:rankcodim1} 
 Assume that $G$ is connected and that every semi-invariant of $\cO(V)$ is invariant. Let $S$ be a codimension one stratum of $Z$. 
 \begin{enumerate}
\item The rank of $dp$ is $\dim Z$ on an open dense subset of $p\inv(S)$.
\item $V\setminus V'$ has codimension at least $2$ in $V$.
\end{enumerate}
\end{lemma}

\begin{proof}  
Let $F=p\inv(p(v))$ where $Gv$ is closed and $p(v)$ lies in $S$.  
Since $\O(V)^G$ is a UFD, the closure of $S$ is defined by an irreducible invariant $f$. Hence $p\inv(\overline{S})$ is irreducible and $df\neq 0$ on an open dense subset of $p\inv(S)$. Now $F$  is isomorphic to $G\times^H\NN(W)$ where $W$ is the slice representation of $H=G_v$. This fiber is the same everywhere over $S$ and $p\inv(S)$ is a fiber bundle over $S$ with fiber $F$. Thus $f\inv(0)$  is schematically  reduced if and only if $F$ is schematically reduced,  i.e., the $G$-invariant polynomials vanishing at $p(v)$ generate the ideal of $F$ in $\O(V)$. Thus at a smooth point of $F$  the rank of $dp$ must be maximal. It follows that $dp$ has maximal rank on an open dense subset of $p\inv(S)$ and we have (1). By Lemma \ref{lem:techcond}(3),  if $T$ is a stratum of $Z$ where $\codim_Z T\geq 2$, then  $\codim_V p\inv(T)\geq 2$. Hence  (2) follows from (1).
 \end{proof}

\begin{cor}\label{cor:Sreduced}
Let $(U,K)$ be a slice representation of $V$ and let $S=(\quot UK)_{(H)}$ be a codimension one stratum where $H\subset K$. Then $p_U\inv(\overline{S})$ is schematically reduced.
\end{cor}

\begin{proof}
Over a point of $S$, the schematic fiber of $p_U$ is $K\times^H\NN(W)$ where $W$ is the slice representation of $H$. The schematic fiber of $p_V$ is   $G\times^H\NN(W)$ over points of  $S':=(\quot VG)_{(H)}$. Thus the ranks of $dp_U$ and $dp_V$ are the same on the inverse images of $S$ and $S'$, respectively, and it follows that  the hypersurface $p_U\inv(\overline{S})$ is schematically reduced. 
\end{proof}

\begin{prop}  \label{prop:ontoZ'} Let $H$ be reductive and $W$ an $H$-module such that the codimension one strata are schematically reduced. Set $Y:=\quot WH$. Suppose that $W'\cap\NN(W)\neq\emptyset$. Then $W$ is coregular and very good, and $p_W(W')=Y$.
\end{prop}

\begin{proof} 
Since $dp_W$ has maximal rank at a point of $\NN(W)$, the image point $0\in Y$ is smooth. But $Y$ has a cone structure (induced by the scalar action of $\C^*$ on $W$). It follows that $Y$ is smooth, i.e., $W$ is a coregular representation of $H$. By Corollary \ref{cor:coreg} we have that $\O(W_m)^{H_m}=p_{W,m}^*\cO(Y_m)$.
Since $W'\cap\NN(W)\neq\emptyset$,  $p_W(W')$ contains a neighborhood of $0\in Y$. Since $W'$ and $Y$ are cones, $p_W(W')=Y$.
\end{proof}

\begin{thm}\label{qiso} Assume that $G$ is connected and that every semi-invariant of $\cO(V)$ is invariant.  Then for all $m\geq 1$, we have $\cO(V_m)^{G_m} =p_m^* \cO(Z'_m)$. \end{thm}

\begin{proof} Let $Gv$ be a closed orbit such that $p\inv(p(v))$ intersects $V'$. Let $(W,H)$ be the slice representation at $v$. Then Proposition \ref{prop:ontoZ'} and Corollary \ref{cor:Sreduced} show that $\cO(W_m)^{H_m}=p_{W,m}^*\cO((\quot WH)_m)$.  Using the slice theorem we see that this implies that  $\cO(U_m)^{G_m}=p_m^*\cO(\tilde Z_m)$ where $U$ is a $G$-saturated neighborhood of $Gv$ and $\tilde Z:=p(U)$ is a neighborhood of $p(v)$. Thus given $f\in\O(V'_m)^{G_m}$ there is a  unique $h\in\cO(Z'_m)$  such that $p_m^*h=f$ on $V'_m$. Since $\codim_VV\setminus V'\geq 2$, $f$ extends to an element of $\O(V_m)^{G_m}$.
\end{proof}

 \section{Classical representations of classical groups}
 
 Let $G=\SL_n$ and $V=k\C^n+\ell(\C^n)^*$. For applications to vertex algebras it would be nice to show that $p_\infty^*\cO(Z_\infty)=\cO(V_\infty)^{G_\infty}$. But we don't know if this is true for general $k$ and $\ell$. On the positive side we are able to show that $p_\infty^*$ is surjective when $k$ or $\ell$ is zero or when $Z$ is a complete intersection. In this section we concentrate on the complete intersection case. We also handle the complete intersection classical representations of the other classical groups (with $\Orth_n$ excluded, since it is not connected).

 Let $G$ be reductive and let $V=V_1\oplus V_2$ be a sum of  $G$-modules.  If $Z_i:=\quot {V_i}G$ is not a complete intersection for some $i$, then neither is $Z=\quot VG$. This is clear because the generators and relations of $\cO(V)^G$ can be chosen to be bihomogeneous in the generators of $\cO(V_i)$, so a minimal set of generators and relations for $\cO(V)^G$ will contain minimal sets of generators and relations for each $\cO(V_i)^G$.

 \begin{lemma}\label{lem:cislice}
 Suppose that $Z$ is a complete intersection. Then $\quot WH$ is a complete intersection for every slice representation $(W,H)$ of $V$.
 \end{lemma} 
 \begin{proof} Let $y$ denote the image of $0\in W$ in $\quot WH$.
By Luna's slice theorem, up to \'etale  mappings (which preserve the property of being a complete intersection), we have an isomorphism of a neighborhood of $y$    with a neighborhood of some $z\in Z$. It follows that $\quot WH$ is a complete intersection in a neighborhood of $y$. Let $p_1,\dots,p_d$ be minimal homogeneous generators of $\O(W)^H$. Then their relations are minimally generated by polynomials $h_j(y_1,\dots,y_d)$ which are homogeneous when we give $y_i$ the degree of $p_i$. Since $\quot  WH$ is a complete intersection near $y$, the number of $h_j$ is   $d-\dim\quot WH$. Hence $\quot WH$ is a complete intersection.
 \end{proof}
 Now consider $(V,G)=(k\C^n,\SL_n)$, $n\geq 3$. If $k=n+2$, then $\cO(V)^G$ has $\binom {n+2}2$ minimal generators  and $\binom {n+2}4$ minimal relations. Its   dimension is $2n+1>\binom {n+2}2-\binom {n+2}4$, which shows that we don't have a complete intersection for $k\geq n+2$. We say that there are \emph{too many relations\/}. For the case $n=2$ we don't have a complete intersection for $k\geq 5$ but we do for $k=4$. Note that here the representations of $\SL_2$ on $\C^2$ and $(\C^2)^*$ are isomorphic.
 
 Let $H_1,\dots,H_r$ be representatives for the conjugacy classes of isotropy groups of closed orbits in $V$. We say that $H_i$ is \emph{maximal proper\/} if $H_i\neq G$ and every other $H_j$ besides $G$ is conjugate to an isotropy group of the slice representation of $H_i$. Of course, in general, such an isotropy group does not exist.

\begin{lemma}\label{lem:maxpropersln}
Let $(V,G)=(k\C^n+\ell(\C^n)^*,\SL_n)$, $n\geq 3$ where $k\ell\neq 0$. Then the maximal proper isotropy group is $\SL_{n-1}$ with slice representation $(k-1)\C^{n-1}+(\ell-1)(\C^{n-1})^*$ (up to trivial factors).
\end{lemma}
\begin{proof}
We may assume that $k\geq\ell$. Write $V=V_1\oplus\cdots\oplus V_\ell\oplus(k-\ell)\C^n$ where each $V_i$ is a copy of $\C^n\oplus (\C^n)^*$. Then   \cite[3.8]{Sch} implies that any proper  isotropy class $(H)$  of $V$ is contained in a proper isotropy class $(L)$ of one of the $V_i$ or $(k-\ell)\C^n$. 
By \cite[Theorem 2.6.A]{We}, the invariants of any $V_i$ are generated by the contraction of $\C^n$ with $(\C^n)^*$. The nonzero   orbits where the invariant does not vanish are closed and have  isotropy class $(\SL_{n-1})$. Again by \cite[Theorem 2.6.A]{We} the nonconstant invariants of $(k-\ell)\C^n$  are  generated by determinants (this only happens for  $k-\ell\geq n$) in which case the  orbits where a determinant does not vanish are closed and have trivial isotropy groups. Hence $\SL_{n-1}$ is maximal proper. Let $W$ be the slice representation of $\SL_{n-1}$. Then we have that
$$
W\oplus\lie{sl}_n/\lie{sl}_{n-1}=V\text{ as $\SL_{n-1}$-modules.}
$$
This gives that $W=(k-1)\C^{n-1}+(\ell-1)(\C^{n-1})^*$ (up to trivial factors).
\end{proof}

Now suppose that $(V,G)$ is as in the lemma above and that $k=n+1$, $n\geq 3$, and that $Z$ is a complete intersection. If $\ell=n$, then there are again too many relations. Thus we must have $\ell<n$. In fact we can have $\ell=(n-1)$. Consider the action of $\SL_{n+1}$ on $\Hom(\C^{n+1},\C^n)\simeq (n+1)\C^n$ and $\SL_{n-1}$ on $\Hom(\C^{n-1},(\C^n)^*)\simeq (n-1)(\C^n)^*$. Then $\SL_{n+1}\times\SL_{n-1}$ acts on the $G$-invariants and the generators transform by the representations $\wedge^n(\C^{n+1})$ (the determinants) and $\C^{n+1}\otimes\C^{n-1}$ (the contractions). By \cite[Theorem 2.14.A]{We}, the relations are generated by those corresponding to 
$$
\C^{n-1}\simeq  \wedge^{n+1}(\C^{n+1})\otimes\C^{n-1}\subset \wedge^n(\C^{n+1})\otimes(\C^{n+1}\otimes\C^{n-1}).
$$ 
We have a complete intersection since 
$$
\binom{n+1}n+n^2-1-(n-1)=n^2+1=\dim Z=\dim V-\dim G=2n^2 -(n^2-1).
$$

Now we state a general theorem enabling us to show that $p_m^*\colon \cO(Z_m)\to\cO(V_m)^{G_m}$ is an isomorphism for all $m\geq 0$. Let $G$ be connected  reductive and $V$ a $G$-module such that every semi-invariant is an invariant and such that $Z$ is a complete intersection. As before, let $Z'$ denote the image of the points $V'$ of $V$ where $dp$ has maximal rank. Let $Z_{m,0}$ denote $Z_m\setminus Z_m'$. Recall that $Z_m'$ consists of smooth points of $Z_m$.

\begin{thm}\label{thm:ci}
Let $G$ and $V$ be as above. Let $(W,H)$ be a nontrivial slice representation and write $W=W^H\oplus W_1$ where $W_1$ is an $H$-module. Let $q\colon {W_1}\to Y:=\quot {W_1}H$ be the quotient mapping and suppose that $\pi_m\inv(q(0))\cap Y_{m,0}$ has codimension at least $2$ in $Y_m$ for all $m\geq 0$ and all $(W,H)$. Then $Z_m$ is normal for all $m$ and $p_m^*\colon \cO(Z_m)\to\cO(V_m)^{G_m}$ is an isomorphism for $m=0,1,\dots,\infty$.
\end{thm}

\begin{proof}
By Boutot \cite{Bo}, it is known that $Z$ has rational singularities. Then \cite{Mu} shows that each $Z_m$ is irreducible, reduced and a complete intersection, and   $Z_\infty$ is reduced and irreducible.  We thus need only show that $Z_m\setminus Z_m'$ has codimension at least two in $Z_m$ for all $m$. Let  $S\subset Z$ be the stratum corresponding to $H$.  Then Luna's slice theorem shows that a neighborhood of $S$ in $Z$ is locally isomorphic  to $S\times Y$ (up to \'etale mappings). Thus $(S\times Y)'_m\simeq S_m\times Y_m'$. Now the points of $Y_{m,0}$ either lie over the stratum $\{q(0)\}$ of $Y$ or they lie over a stratum $T\subset Y$ with a smaller isotropy group. By induction we can assume that $\pi_m\inv(T)\cap Y_{m,0}$ has codimension at least two in $Y_m$. By hypothesis, $\pi_m\inv(q(0))\cap Y_{m,0}$ has codimension at least two in $Y_m$. Thus $Y_{m,0}$ has codimension two in $Y_m$ and $Z_{m,0}$ has codimension two in $Z_m$ over a neighborhood of $S$. Since this is true for all strata $S$, we find that $Z_{m,0}$ has codimension two in $Z_m$, i.e., $Z_m$ is normal. Thus $p_m^*\colon \cO(Z_m)\to\cO(V_m)^{G_m}$ is an isomorphism for  $m \leq \infty$.
\end{proof}

\begin{thm}\label{thm:p=n+1sln}
Let $(V,G)=((n+1)\C^n+(n-1)(\C^n)^*,\SL_n)$, $n\geq 2$. Then $Z_m$ is normal for all $m\geq 0$ and $p_m^*\colon \cO(Z_m) \to \cO(V_m)^{G_m}$ is an isomorphism for $m=0,1,\dots,\infty$.
\end{thm}

\begin{proof}
Let us call our $n-1$ relations $f_1,\dots,f_{n-1}$. Then the $f_j$ are bilinear, being linear in the determinants and the contractions. Let 
$$
(\omega(t),\alpha(t))=(\sum_{i=1}^m t^i\omega_i,\sum_{i=1}^mt^i\alpha_i)
$$ 
be elements of $\pi_m\inv(p(0))$ where the $\omega_i$ correspond to the    determinants and the $\alpha_i$ to the contractions. Then for $j=1,\dots,n-1$ we have the equations $f_j(\omega(t),\alpha(t))=0\mod t^{m+1}$. Thus we get the equations
\begin{equation*}
\begin{split}
f_j(\omega_1,\alpha_1)=0,\  &f_j(\omega_2,\alpha_1)+f_j(\omega_1,\alpha_2)=0,\dots\\
&\dots,f_j(\omega_{m-1},\alpha_1)+\dots+f_j(\omega_1,\alpha_{m-1})=0.
\end{split}
\end{equation*}
We have no conditions on $\omega_m$ and $\alpha_m$, and the equations above on the $\omega_i$ and $\alpha_{i'}$ for $1\leq i,i'\leq m-1$ give something  isomorphic to  $Z_{m-2}$. Thus the dimension of $\pi_m\inv(p(0))$ is $\dim Z_{m-2}+(n^2+n)$ for $m\geq 2$ and the codimension of $\pi_m\inv(p(0))$ in $Z_m$ is 
\begin{equation*}
\begin{split}
&\dim Z_m-(\dim Z_{m-2})-(n^2+n)\\
&=(m+1-(m-1))(n^2+1)-(n^2+n)=n^2-n+2\geq 4.
\end{split}
\end{equation*}
For $m=1$ we get the same codimension and for $m=0$ the codimension is $n^2+1\geq 5$. Now we can use Lemma \ref{lem:maxpropersln} and Theorem \ref{thm:ci} to finish the proof.
\end{proof}

We consider later the case of $(n(\C^n+(\C^n)^*),\SL_n)$ whose quotient is a  complete intersection.

\begin{thm}
Let $(V,G)=(k\C^{2n},\Sp_{2n})$, $n\geq 1$. Then $V$ is coregular for $k\leq 2n+1$, $Z$ is a hypersurface for $k=2n+2$ and is not a complete intersection for $k\geq 2n+3$. When $k=2n+2$, $Z_m$ is normal for all $m$ and  $p_m^*\colon \cO(Z_m)\to \cO(V_m)^{G_m}$ is an isomorphism for $m=0,1,\dots,\infty$.
\end{thm}

\begin{proof} The group $G$ preserves a canonical non-degenerate skew form on $\C^{2n}$.
If we use indices $1\leq i<j\leq k$ for pairs of copies of $\C^{2n}$, then the invariants of $V$ give a 2-form $\omega=\sum \omega_{ij} e_i\wedge e_j $ and the relations of the $\omega_{ij}$ are given by the vanishing of   $\omega^{n+1}\in\wedge^{2n+2}(\C^k)\otimes\C[\omega_{ij}]$. It follows that $V$ is coregular for $k\leq 2n+1$ and has too many relations for $k\geq 2n+3$. Thus we need only consider the hypersurface case $k=2n+2$.
The noncoregular slice representations of $V$ are of the same form as $V$ with $n$ replaced by a smaller $n'\geq 1$, modulo trivial representations.  We show that $Z_{m,0}$ has codimension 4 in $Z_m$ which, by Theorem \ref{thm:ci}, establishes our result.

First consider the case where $m$ is at least $n+1$. Let 
$$
\omega=\sum_{i=1}^mt^i\omega_i\in \pi_m\inv(p(0)).
$$ 
Here the $\omega_i$ are elements of $\bigwedge^2(\C^{2n+2})_i$. Then we have the  equations
$\omega_1^{n+1}=0$, $\omega_1^n\wedge \omega_2=0,\dots$.
The equations applied to $n+1$ of the $\omega_i$ such that the sum of the indices does not exceed $m$  gives a set isomorphic to $Z_{m-n-1}$, and there are no conditions on $\omega_{m-n+1},\ \omega_{m-n+2},\dots,\omega_m$. Now $Z$ has dimension $\binom{2n+2}2-1$ and it follows that the codimension of $\pi_m\inv(p(0))$ in $Z_m$ is
\begin{equation*}
\begin{split}
&((m+1)-(m-n))(\binom{2n+2}2-1)-n\binom{2n+2}2\\
&=\binom{2n+2}2-n-1=2n^2+2n\geq 4.
\end{split}
\end{equation*}
For the cases $m=0,\dots,n$ one gets that
$\codim \pi_m\inv(p(0))=\dim Z-m$ which is even better. This completes the proof.
\end{proof}

\begin{thm}
Let $(V,G)=(k\C^n+\ell(\C^n)^*,\GL_n)$, $n\geq 1$, $k\geq \ell$. Then $V$ is coregular if $\ell\leq n$, there are too many relations if $\ell\geq n+1$ and $k+\ell\geq 2n+3$  and $Z$ is a hypersurface if $k=\ell=n+1$. In this last case we have that $Z_m$ is normal for all $m$ and $p_m^*\colon \cO(Z_m)\to \cO(V_m)^{G_m}$ is an isomorphism for $m=0,1,\dots,\infty$.
\end{thm}

\begin{proof}
The invariants are just the contractions $\alpha_{ij}$, $1\leq i \leq k$, $1\leq j\leq \ell$ and the relations are $\det_{r,s=1}^{n+1}\alpha_{i_r,j_s}=0$ where the $i_r$ are distinct and the $j_s$ are distinct. This shows that $V$ is coregular for $\ell\leq n$ and that $Z$ is not a complete intersection for $\ell\geq n+2$ or $\ell=n+1$ and $k>n+1$. We now consider the case $k=\ell=n+1$. Then all the nontrivial slice representations have the same form, with $n$ replaced by a smaller $n'\geq 1$. The codimension of $p(0)$ in $Z_0$ is $\dim Z=(n+1)^2-1\geq 3$. For $m\geq n+1$ one can use the techniques as above to show that the codimension of $\pi_m\inv(p(0))$ in $Z_m$ is at least 
\begin{equation*}
\begin{split}
&\dim Z_m-\dim Z_{m-n-1}-n(n+1)^2\\
&=((m+1)-(m-n))((n+1)^2-1)-n(n+1)^2=n(n+1)\geq 2.
\end{split}
\end{equation*}
For $1\leq m\leq n$ the codimension is at least as great.  Thus, as before, we find that $Z_m$ is normal for all $m$ and the theorem follows.
\end{proof}

\begin{thm}
Let $(V,G)=(k\C^n,\SO_n)$, $n\geq 2$. Then $V$ is coregular for $k<n$ and has too many relations if $k>n$. If $k=n$, then $Z$ is a hypersurface and $Z_m$ is normal for all $m$. Hence $p_m^*\colon \cO(Z_m)\to \cO(V_m)^{G_m}$ is an isomorphism for $m=0,1,\dots,\infty$.
\end{thm}

\begin{proof} We leave it to the reader to show that there are too many relations if $k>n$ and that $V$ is coregular for $k<n$. Consider the case that $k=n$. The invariants of $V$ are generated by the inner products $\alpha_{ij}$, $1\leq i,j\leq n$ and the determinant  $d$. The relation  is $\det_{i,j=1}^n\alpha_{ij}=d^2$. 
The nontrivial slice representations of $V$ are just those for $n$ replaced by $n'$ where $2\leq n'<n$, up to trivial factors, so it is enough to show that  $Y_m:=\pi_m\inv(p(0))$ has codimension at least 2 in $Z_m$ for $m\geq 0$. Let $\alpha(t)=\sum_{i=1}^m t^i\alpha_i$ and $d(t)=\sum_{i=1}^m t^id_i$ be elements of $Y_m$ where $m\geq n$. 
For now suppose that $n=2\ell$ is even. Then the equations $\det(\alpha(t))=d(t)^2$ force $d_1=\dots=d_{\ell-1}=0$. The  $n$th equation is $\det(\alpha_1)=d_{\ell}^2$ and one can see that the dimension of $Y_m$ is $\dim Z_{m-n}+(n-1)(\dim Z+1)-(n-2)/2$ so that the codimension of $Y_m$ in $Z_m$ is
$$
\dim Z-(n-1)+(n-2)/2=\frac 12 n(n+1)-n+1+\frac 12 (n-2)=\frac 12n^2
$$
which is at least $2$. If $m<n$, then we get the estimate that $\codim Y_m\leq \dim Z-m+[\frac m2]$ which is even better.

Now we consider the case that $n=2\ell+1$ is odd.
Then looking at the coefficient of $t^n$ we get the equation $\det(\alpha_1)=0$ since the equations with lower powers of $t$ force $d_1=\dots=d_{\ell}=0$. The coefficient of $t^{n+1}$ gives an inhomogeneous equation for $\alpha_1$ and $\alpha_2$ with right hand side $d_{\ell+1}^2$. But the solution space of the inhomogeneous equation has dimension at most that of the homogeneous equation. Hence we can estimate the dimension of $Y_m$ by replacing the inhomogeneous equations by the homogeneous equations. Thus we get the estimate $\dim Y_m=\dim Z_{m-n}+(n-1)(\dim Z+1)-(n-1)/2$ and for the codimension we get
$$
\dim Z-(n-1)+\frac{n-1}2=\frac 12 n(n+1)-n+1+\frac 12(n-1)=\frac 12 n^2+\frac 12.
$$  
As before, the estimates for $0\leq m<n$ are even better, so $Y_m$ has codimension 2 and $Z_m$ is normal.
\end{proof}

There is one case left, which needs no new techniques.

\begin{thm}
Let $(V,G)=(n\C^n+n(\C^n)^*,\SL_n)$, $n\geq 2$. Then $Z_m$ is normal and $p_m^*\colon \cO(Z_m) \to \cO(V_m)^{G_m}$  is an isomorphism for $m=0,1,\dots, \infty$.
\end{thm}

\begin{proof}
The generators of $\cO(V)^G$ are the contractions $\alpha_{ij}$ and the determinants $d$ and $e$ of the $n$ copies of $\C^n$ and its dual. The relation is $\det(\alpha_{ij})=de$. The nontrivial slice representations, up to trivial factors, have the same form, with $n$ replaced by a smaller $n'\geq 2$. As above, one can show that $\pi_m\inv(p(0))$   always has codimension at least $4$ in $Z_m$ which gives the result.
\end{proof}

\section{Representations of $\C^*$}\label{sec:cstar}
 
The main theorem of this section is the following.

\begin{thm}\label{thm:cstar}
Let $G=\C^*$ and let $V$ be a $G$-module all of whose weights are $\pm 1$. Then $p_\infty^*\colon \cO(Z_\infty)\to\cO(V_\infty)^{G_\infty}$ is an isomorphism and $\cO(Z_\infty)$ is integrally closed.
\end{thm}

\begin{cor}
Let $V$ be a $G=\C^*$ module with nonzero positive weights $r_1\leq\dots\leq  r_d$, $d\geq 1$, and nonzero negative weights $-s_1\leq\dots\leq -s_e$, $e\geq 1$, where we may assume that the greatest common divisor  of the $r_i$ and $s_j$ is one. Then $p_\infty^*$ is surjective only in the following cases.
\begin{enumerate}
\item $d=1$ and $r_1=1$ or  $e=1$ and $s_1=1$. Here $Z'=Z$ so that $p_m^*$ is an isomorphism for all $m$.
\item All the $r_i$ and $s_j$ are $1$.
\end{enumerate}
\end{cor}

\begin{proof}
Suppose that there are weights $r_i$ and $-s_j$ neither of which is 1. Let $W$ be the corresponding two dimensional submodule of $V$. Then the null cone of $W$ consists of two non-reduced hypersurfaces, hence $W$ is not good and neither is $V$. Thus, without loss of generality, we can assume  that $r_i=1$ for all $i$. If $d=1$, then we are in case (1) and one easily sees that $Z'=Z$. Suppose that $d>1$. If  $s_j=s\neq 1$ for some $j$, then $V$ contains a submodule $W$ with weights 1, 1 and $-s$. Let  $x_1$, $x_2$ and $y$ be the corresponding coordinate functions. The generators of the invariants have degree $s+1$. Now $(x_1Dx_2-x_2Dx_1)x_1^{s-2}y$ is an invariant of $G_\infty$ of weight one and degree $s+1$ but is not $D$ applied to a  generator  of $\cO(W)^{\C^*}$. Hence $W$ is not good. Thus we have to be in case (2).
\end{proof}

We now prove the theorem. We may assume that $V^G=(0)$ and that the weights $1$ and $-1$ have multiplicity $n$. Let $x_1,\dots,x_n$ be coordinate functions corresponding to the positive weights and let $y_1,\dots,y_n$ be coordinate functions corresponding to the negative weights.  Set $x_i^{(k)}:=D^kx_i$ and $y_j^{(k)}=D^k y_j$ for $k\geq 0$. Then the $x_i^{(k)}$ and $y_j^{(k)}$ are coordinate functions on $V_\infty$. We order the $x_i^{(k)}$ so that $x_i^{(k)}<x_j^{(\ell)}$ if $k<\ell$ or $k=\ell$ and $i<j$. We similarly order the $y_i^{(k)}$. The $G$-invariants of $R:=\cO(V_\infty)$ are linear combinations of monomials $W:=u_1\cdots u_r v_1\cdots v_r$ where each $u_i$ is an $x_j^{(k)}$ and each $v_i$ is some $y_{j'}^{(k')}$ and we have that $u_1\leq \dots\leq u_r$ and $v_1\leq\dots\leq v_r$.  We say that $W$ is a word in \emph{standard form}. We define the weight $\wt(W)$ of $W$ to be the sums of the exponents of the $x_i^{(k)}$ and $y_j^{(\ell)}$ occurring in $W$. Let $W'=u_1'\cdots u_r'v_1'\cdots v_r'$ be another monomial in standard form.  We say that  $W<W'$ if
\begin{enumerate}
\item $\wt(W)<\wt(W')$,
\item $\wt(W)=\wt(W')$, $u_1\cdots u_sv_1\cdots v_s=u_1'\cdots u_s'v_1'\cdots v_s'$ and $u_{s+1}<u_{s+1}'$ where $0\leq s<r$, or
\item $\wt(W)=\wt(W')$, $u_1\cdots u_{s+1}v_1\cdots v_s=u_1'\cdots u_{s+1}'v_1'\cdots v_s'$ and   $v_{s+1}\\<v_{s+1}'$ where $0\leq s<r$.
\end{enumerate}
We say that $W$ is \emph{admissible\/} if it is in standard form, $v_1$ is some $y_j$ and for $1\leq s<r$, $D(v_s)>v_{s+1}$.
Given  $0\neq h\in\cO(V_\infty)^{G_\infty}$ of degree $2r$ and weight $m$ it is a sum $\sum_ic_iW_i$ where the $W_i$ are distinct and in standard form, the $c_i$ are nonzero scalars and $\wt(W_i)=m$ and $\deg(W_i)=2r$ for all $i$. We define the leading term $M(h)$ to be the greatest $W_i$. We will eventually show that $M(h)$ is admissible.

The ring $\cO(Z)$ has generators $f_{ij}$ (corresponding to $x_iy_j$) and relations $R_{ab,cd}=f_{ab}f_{cd}-f_{ad}f_{cb}$, $1\leq a,b,c,d\leq n$ and $a\neq c$, $b\neq d$.  
Then $\cO(Z_\infty)$ has generators $f^{(k)}_{ij}:=D^kf_{ij}$ and relations generated by the $D^kR_{ab,cd}$. Of course, $p_\infty^*f^{(k)}_{ij}=D^k(x_iy_j)$. We define a partial order on the $f^{(k)}_{ij}$ where $f_{ij}^{(k)}\leq f_{i'j'}^{(k')}$ if 
\begin{enumerate}
\item $k+2\leq k'$,
\item $k+1=k'$ and $i\leq i'$ or $j\leq j'$, or 
\item $k=k'$ and $i\leq i'$ and $j\leq j'$.
\end{enumerate}
We say that a monomial $f_{i_1,j_1}^{(k_1)}\cdots f_{i_r,j_r}^{(k_r)}$ is \emph{standard\/} if $f_{i_1,j_1}^{(k_1)} \leq\dots \leq f_{i_r,j_r}^{(k_r)}$.

\begin{lemma}
The algebra $\cO(Z_\infty)$ is spanned by the standard monomials.
\end{lemma}
  
  \begin{proof}
Suppose that $f_{ab}^{(k)}f_{cd}^{(\ell)}$ and $f_{cd}^{(\ell)}f_{ab}^{(k)}$ are not standard. First suppose that $k=\ell$. We may assume that  $a>c$ and $b<d$. Consider the relation $D^{2k}(f_{ab}f_{cd}-f_{ad}f_{cb})=0$. It is a sum $\sigma+f_{ab}^{(k)}f_{cd}^{(k)}-f_{ad}^{(k)}f_{cb}^{(k)}$ where the terms in $\sigma$ are standard and involve an $f$ with weight higher than $k$. We   replace $f_{ab}^{(k)}f_{cd}^{(k)}$ by   $f_{ad}^{(k)}f_{cb}^{(k)}-\sigma$. The new expression is a sum of standard terms and is ``larger'' than $f_{ab}^{(k)}f_{cd}^{(k)}$ in that each term has a factor with a higher weight or contains the term $f^{(k)}_{ad}$ which is larger than $f^{(k)}_{ab}$ and $f^{(k)}_{cd}$ in our partial order.

  Now suppose that $\ell=k+1$ and our terms are not standard. Then $a>c$ and $b>d$. Consider the relation $D^{2k+1}(f_{ab}f_{cd}-f_{ad}f_{cb})=0$. Expanding we get a sum 
  $$
  \sigma+f_{ab}^{(k)}f_{cd}^{(k+1)}+f_{ab}^{(k+1)}f_{cd}^{(k)}-f_{ad}^{(k)}f_{cb}^{(k+1)}-f_{ad}^{(k+1)}f_{cb}^{(k)} 
  $$
  where $\sigma$ is a sum of standard terms with a factor of weight at least $k+2$. Now the term $f_{ab}^{(k+1)}f_{cd}^{(k)}$ is standard and the factor $f_{ab}^{(k+1)}$ is larger than $f^{(k)}_{ab}$ and $f^{(k+1)}_{cd}$ in our partial order. Similarly, $f_{ad}^{(k)}f_{cb}^{(k+1)}$ is standard with $f_{cb}^{(k+1)}$ larger than $f^{(k)}_{ab}$ and $f^{(k+1)}_{cd}$. Similarly for $f_{ad}^{(k+1)}f_{cb}^{(k)} $. Hence we can replace $f_{ab}^{(k)}f_{cd}^{(k+1)}$ by larger standard terms.
  
  Now it follows by   induction  that any monomial in the $f_{ij}^{(k)}$ is a sum of standard monomials.
   \end{proof}
  
  Let $w=f_{i_1,j_1}^{(k_1)}\cdots f_{i_r,j_r}^{(k_r)}$ be standard. We construct a word 
  $$
  L(w)=u_1\cdots u_rv_1\cdots v_r
  $$   
  such that
  \begin{enumerate}
  \item $u_1\leq\dots\leq u_r$ and $v_1\leq\dots\leq v_r$, i.e., $L(w)$ is a standard word.
\item $u_sv_s$ is a monomial occurring in $p_\infty^*(f^{(k_s)}_{i_s,j_s})$, $1\leq s\leq r$.
\item  $v_1$ is a  $y_i$ and $Dv_s>v_{s+1}$, $1\leq s<r$, i.e., $L(w)$ is admissible.
\end{enumerate}
Clearly we must have that $u_1v_1=x_{i_1}^{(k_1)}y_{j_1}^{(0)}$. Suppose that we have determined $u_1,\dots,u_{r-1}$ and $v_1,\dots,v_{r-1}$ satisfying (1)--(3) for $r$ replaced by $r-1$. We have that $u_{r-1}=x_{i_{r-1}}^{(a)}$ and $v_{r-1}=y_{j_{r-1}}^{(b)}$ where $a+b=k_{r-1}$. Suppose  that $k_r=k_{r-1}$. Then $i_{r-1}\leq i_r$ and $j_{r-1}\leq j_r$ so   $u_r=x_{i_r}^{(a)}$ and $v_r=y_{j_r}^{(b)}$ satisfy our conditions. If $k_r=k_{r-1}+1$ and $i_{r-1}>i_r$, then $j_{r-1}\leq j_r$ and we set $u_r=x_{i_r}^{(a+1)}$ and $v_r=y_{j_r}^{(b)}$. The case where $j_{r-1}>j_r$ is similar. If $i_{r-1}\leq i_r$ and $j_{r-1}\leq j_r$, then we set $u_r=x_{i_r}^{(a+1)}$, $v_r=y_{j_r}^{(b)}$. The case $k_r\geq 2+k_{r-1}$ is handled similarly.  

\begin{prop}\label{prop:L(w)maximal}
Let $w$ be as above. Then $L(w)=M(p_\infty^*(w))$.
\end{prop}

\begin{proof}
We have that $L(w)=u_1\cdots u_rv_1\cdots v_r$ where $u_s=x_{i_s}^{(a_s)}$ and $v_s=y_{j_s}^{(b_s)}$ and $a_s+b_s=k_s$, $1\leq s\leq r$. Now $M:=M(p_\infty^*(w))=u_1'\cdots u_r'v_1'\cdots v_r'$ where each $u_s'$ is some $x_{i_{s}'}^{a_{s}'}$ and each $v_{s}'$ is some $y_{j_{s}'}^{b_{s}'}$. 

First we prove that $u_1'=u_1$ and that  $v_1'=v_1$ where $u_1=x_{i_1}^{(k_1)}$  and $v_1=y_{j_1}^{(0)}$.
Since $f_{i_1,j_1}^{(k_1)}$ is a factor of $w$, some $u_s'$ is $x_{i_1}^{a}$  where $a\leq k_1$. Hence $u_1'\leq u_1$ and by maximality of $M$ we must have equality. It follows that the monomial of $p_\infty^*(f_{i_1,j_1}^{(k_1)})$ which occurs in $M$ is $x_{i_1}^{(k_1)} y_{j_1}^{(0)}$. Hence some $v_s'$ is $y_{j_1}^{(0)}$ and $v_1'\leq y_{j_1}^{(0)}$. By maximality of $M$ we have $v_1'=v_1$.

Now suppose by induction that $u_i=u_i'$ and $v_i=v_i'$ for $i<r$. Then $u_r'v_r'$ is a monomial $x_{i_r}^{(a_r')} y_{j_r}^{(b_r')}$ occurring in $p_\infty^*(f_{i_r,j_r}^{(k_r)})$ where $a_r'+b_r'=k_r$. If $j_{r-1}\leq j_r$, then the largest possible $a_r'$ is $a_r=k_r-b_{r-1}$ and we have that $M=L(w)$. If $j_{r-1}>j_r$, one has $a_r'=a_r=k_r-b_{r-1}-1$ and again $M=L(w)$. Hence we always have $L(w)=M(p_\infty^*(w))$.
\end{proof}

If $W$ is  admissible, then it is clear that there is a unique $w$ with $W=L(w)$, hence we have the following

 \begin{lemma}
  The function $w\mapsto L(w)$ is  a bijective mapping from  standard monomials to admissible monomials. Hence the standard monomials are linearly independent.
  \end{lemma}
  Since the smallest possible $u_1v_1$ is $f_{11}=x_1y_1$, multiplication by $f_{11}$ is injective on linear combinations of standard monomials. Hence $f_{11}$ is not a zero divisor in $\cO(Z_\infty)$ and the mapping from $\cO(Z_\infty)$ to its localization  $\cO(Z_\infty)_{f_{11}}$ is injective. By Corollary \ref{cor:localization} this localization is isomorphic to  $(\cO(V_\infty)^{G_\infty})_{(x_1y_1)}$. Hence we have

  \begin{prop}
Let $G = \mathbb{C}^*$ and let $V$ be a $G$-module with weights $\pm 1$ of multiplicity $n$, as above. The mapping $p_\infty^*\colon \cO(Z_\infty)\to\cO(V_\infty)^{G_\infty}$ is injective, hence  $\cO(Z_\infty)$ is reduced. Given $h\in\cO(V_\infty)^{G_\infty}$ there  is an $s\geq 0$ and an $f\in\cO(Z_\infty)$ such that $(x_1y_1)^sh=p_\infty^*(f)$.
  \end{prop}

\begin{cor}
Let $h\in\cO(V_\infty)^{G_\infty}$. Then $M(h)$ is admissible.
\end{cor}

\begin{proof}
There is an $s\geq 0$ such that $(x_1y_1)^sh$ is in the image of $\cO(Z_\infty)$. Hence $M((x_1y_1)^sh)$ is admissible. But $M((x_1y_1)^sh)=(x_1y_1)^sM(h)$ and hence $M(h)$ is admissible.
 \end{proof}

\begin{proof}[Proof  of Theorem \ref{thm:cstar}]
 Let $h\in\cO(V_\infty)^{G_\infty}$ have a fixed degree and weight. Let $W=M(h)$. Then, as we saw before, there is a canonical standard monomial $w$ such that $M(p_\infty^*(w))=W$. Then for some constant $c$, $M(h-p_\infty^*(cw))<M(h)$. By induction, then, we get that $h\in p_\infty^*\cO(Z_\infty)$.
 Since the group $G_\infty$ is connected, one shows as usual that $\cO(V_\infty)^{G_\infty}$ is integrally closed, hence so is $\cO(Z_\infty)$.
\end{proof}

\section{Some representations of $\SL_n$}
In this section we consider the case where $(V,G)=(\ell\C^n,\SL_n)$. We use a version of standard monomial theory to prove that $p_\infty^*\cO(Z_\infty)=\cO(V_\infty)^{G_\infty}$. For $n=2$ one can show that $p_\infty^*$ is injective, but for $n\geq 3$ this fails, as we show in Example \ref{ex:sl3}.

 Consider pairs $(j,k)$ where $1\leq j\leq \ell$ and $k\geq 0$. For each $j$ let $\{x_i^{(j,0)}\}_{i=1}^n$ denote the usual coordinate functions on the $j$th copy of $\C^n$. Let $x_i^{(j,k)}$ denote $(1/k!)D^kx_i^{(j,0)}$. Then the $x_i^{(j,k)}$ are coordinate functions on $V_\infty$. Let $\omega=(j,k)$. Then we write $\omega<\omega'=(j',k')$ if $k<k'$ or $k=k'$ and $j<j'$. We write $x_i^\omega<x_{i'}^{\omega'}$ if $\omega<\omega'$ or $\omega=\omega'$ and $i'<i$. Thus we have $x_1^\omega>\dots>x_n^\omega$ for any $\omega$.
We inductively define a monomial order as follows. Let $M=mx_i^\omega$ (resp.\ $M'=m'x_{i'}^{\omega'}$) be a monomial where $x_i^\omega$ (resp.\ $x_{i'}^{\omega'}$) is the largest variable occurring in $M$ (resp.\ $M'$).  If $\deg(M)<\deg(M')$ or $\deg M=\deg(M')$ and $\wt(M)<\wt(M')$, then $M<M'$. If $M$ and $M'$ have the same degree and weight, then $M<M'$ if $x_i^{\omega}<x_{i'}^{\omega'}$ or $x_i^\omega=x_{i'}^{\omega'}$ and $m<m'$. 
 If $h\neq 0$ is in $\cO(V_\infty)^{G_\infty}$, we let $L(h)$ denote the lowest monomial occurring with nonzero coefficient in the expression of $h$ in the $x_i^\omega$.
  
Let $x^\omega$ denote the vector $(x_1^\omega,\dots,x_n^\omega)$.
By classical invariant theory, the algebra $\cO(V_\infty)^{G}$ is generated by determinants $[x^{\omega_1},\dots,x^{\omega_n}]$.

\begin{remarks}
\begin{enumerate}
\item  We have $L([x^{\omega_1},\dots,x^{\omega_n}])=x_1^{\omega_1}\cdots x_n^{\omega_n}$ where $\omega_1<\dots<\omega_n$.
\item Let $1\leq j_1<\dots<j_n\leq \ell$ and let $k=na+b$ where $a$, $k\in\ZZ^+$ and $0\leq b<n$. Then 
\begin{equation*}
\begin{split}
&L(D^k([x^{(j_1,0)},\dots,x^{(j_n,0)}]))\\
&=L([x^{(j_{b+1},a)},\dots,x^{(j_n,a)},x^{(j_1,a+1)},\dots,x^{(j_b,a+1)}]).
\end{split}
\end{equation*}
\end{enumerate}
\end{remarks}

Let 
\begin{equation}\label{eq:W}
W=\left(\begin{matrix} \omega_{1,1} & \omega_{1,2} & \dots & \omega_{1,n}\\
\vdots & \vdots & \omega_{i,j} & \vdots\\
\omega_{s,1} & \omega_{s,2} &\cdots &\omega_{s,n}\end{matrix}\right)
\end{equation}
where the $\omega_{i,j}$ are pairs as above.
 We call $W$ a \emph{tableau\/} and we say that $W$ is \emph{standard\/} if the rows are strictly increasing and the columns are nondecreasing. To each row $\omega_{r,1},\dots,\omega_{r,n}$ of $W$ we associate the determinant $[x^{\omega_{r,1}},\dots,x^{\omega_{r,n}}]$ and to $W$ we associate the product $P(W)$ of the determinants determined by the rows. Let $Q(W)$ denote $\prod_s\prod_{t=1}^n x_t^{\omega_{s,t}}$.

 \begin{prop}\label{prop:stdmonom}
 \begin{enumerate}
\item The  $P(W)$ for $W$ a standard tableau are a basis of $\cO(V_\infty)^G$.
\item The mapping $W\mapsto Q(W)$ is injective and $Q(W)=L(P(W))$.
\item  If $0\neq h\in\cO(V_\infty)^{G}$, then $L(h)=Q(W)$ for a standard $W$.
\end{enumerate}
 \end{prop}
\begin{proof}
Part (1) is just standard monomial theory and (2) is obvious. Let $h$ be as in (3). Then $h$ is a sum $\sum_i c_iP(W_i)$ where the $c_i$ are nonzero and the $W_i$ are standard and distinct. Then the $Q(W_i)$ are distinct and $L(h)$ is the minimum of the $Q(W_i)$.
\end{proof}

For $\omega=(j,k)$ with $\wt(\omega)=k>0$, let $\tilde\omega$ denote $(j,k-1)$. Let $X\in \gg = \gs\gl_n$ be the  element which sends $x_n\to x_1$ and annihilates $x_i$ for $i<n$. Here the $x_i$ are the usual coordinate functions on $\C^n$. Let $F$ denote $tX\in\lieg[t]$. Then $F$ and $\lieg$ generate $\lieg[t]$ so that $\cO(V_\infty)^{G_\infty} = \cO(V_\infty)^{\lieg[t]}=\cO(V_\infty)^G\cap\cO(V_\infty)^F$. Now $F$ annihilates the $x_n^{(j,0)}$ and the $x_i^\omega$ for $i\neq n$, and it sends $x_n^\omega$ to $x_1^{\tilde\omega}$ when $\wt(\omega)>0$. For $\wt(\omega)>0$ let $y_2^\omega$ denote $x_n^\omega$ and let $y_1^\omega$ denote $x_1^{\tilde\omega}$. Then the action of $F$ on the $y_i^\omega$ is the standard action of the Lie algebra of the maximal unipotent subgroup $U$ of $\SL_2$. Consider the symmetric algebra $A$  in the $y_i^\omega$ for $\wt(\omega)>0$. For fixed $\omega$,  the invariants of $U$ acting on the subalgebra generated by $y_1^\omega$ and $y_2^\omega$ are generated by $y_1^\omega$. Then it follows from \cite[Theorem 2.5.A]{We} that the $U$-invariants of $A$ are generated by the $y_1^\omega$ and the determinants $I^{\omega}_{\omega'}:=[y^\omega,y^{\omega'}]=y_1^\omega y_2^{\omega'}-y_1^{\omega'}y_2^\omega$.   Then   we have the following relations:
$$
I^{\omega_1}_{\omega_2}=-I^{\omega_2}_{\omega_1},\quad I^{\omega_1}_{\omega_2}I^{\omega_3}_{\omega_4}=I^{\omega_3}_{\omega_2}I^{\omega_1}_{\omega_4}+I^{\omega_1}_{\omega_3}I^{\omega_2}_{\omega_4},\quad I^{\omega_1}_{\omega_2}y_1^{\omega_3}=I^{\omega_3}_{\omega_2}y_1^{\omega_1}+I^{\omega_1}_{\omega_3} y_1^{\omega_2}
$$
where $\wt(\omega_i)>0$, $i=1,\dots,4$.   Let 
\begin{equation}\label{eq:Y}
Y=\left(\begin{matrix} \omega_{1} & \omega_{1}'\\
\vdots & \vdots\\
\omega_{q} & \omega_{q}'\\
\omega_{q+1} & \\
\vdots & \\
\omega_{q+r}
\end{matrix}\right)
\end{equation}
be an array of our pairs $(j,k)$ (a tableau) where $k>0$. As usual, we say that $Y$ is \emph{standard\/} if the pairs are strictly increasing in the rows and nondecreasing as one goes down the columns. Let $\widetilde P(Y)$ denote the polynomial $[y^{\omega_1},y^{\omega_1'}]\cdots [y^{\omega_q},y^{\omega_q'}]y_1^{\omega_{q+1}}\cdots y_1^{\omega_{q+r}}$. We say that $\widetilde P(Y)$ is a \emph{standard monomial}. Then the relations above show that $A^F=A^U$ has basis the standard monomials.

We have the induced monomial order on $A\subset \cO(V_\infty)$. Then for $\omega<\omega'$ we have
$$
L([y^{\omega},y^{\omega'}])=y_1^{\omega'}y_2^{\omega}=x_1^{\tilde\omega'}x_n^{\omega}.
$$
\begin{lemma}
Let $Y$ be as above. Then the leading term of $L(\widetilde P(Y))$ is 
\begin{equation} \label{ymonomial}
y_1^{\omega_{1}'}y_2^{\omega_{1}}\cdots y_1^{\omega'_q}y_2^{\omega_q}y_1^{\omega_{q+1}}\cdots y_1^{\omega_{q+r}}=x_1^{\tilde\omega_1'}x_n^{\omega_1}\cdots x_1^{\tilde\omega_q'}x_n^{\omega_q}x_1^{\omega_{q+1}}\cdots x_1^{\omega_{q+r}}.
\end{equation}
\end{lemma}

Let $A'$ be the subalgebra of $\cO(V_\infty)$ generated by $A$ and the $x_n^{(j,0)}$, $1\leq j\leq \ell$, and let $B$ be the subalgebra of $\cO(V_\infty)$ generated by the $x_i^\omega$ for $1<i<n$. Then $\cO(V_\infty)$ is the tensor product $A'\otimes_\C B$.

\begin{cor}\label{cor:main}
\begin{enumerate}
\item Let $h\in (A')^F$. Then $L(h)$ is of the form $fL(\widetilde P(Y))$ where $Y$ is standard and $f$ is a homogeneous polynomial in the $x_n^{(j,0)}$.
\item Let $h\in\cO(V_\infty)^F$. Then $L(h)$ is of the from $fL(\widetilde P(Y))$ where $Y$ is standard and $f$ is a homogeneous polynomial in the $x_n^{(j,0)}$ and $x_t^\omega$ for $1<t<n$.
\item Let $h\in\cO(V_\infty)^{G_\infty}$. Then $L(h)=L(P(W))$ where $W=(\omega_{ij})$ is as in \eqref{eq:W}. For any $i$ with $1\leq i\leq s$, if $\wt(\omega_{i,n})>0$, then $\tilde\omega_{i,n}<\omega_{i,1}$.
\end{enumerate}
\end{cor}

\begin{proof}
Parts (1) and (2) are immediate. Let $h$ be as in (3). Then we know that there is a standard $W=(\omega_{ij})$ as in \eqref{eq:W} such that $L(h)=L(P(W))$. On the other hand, we also know that there is a standard $Y$ as in \eqref{eq:Y} such that $L(h)=f(x_n^{(j,0)},x_t^\omega)L(\widetilde P(Y))$ where $f$ is homogeneous. Suppose that $\wt(\omega_{i,n})>0$ for  $i>i_0$. Then for $i>i_0$  there is an $\omega_{i',1}$ such that $\tilde\omega_{i,n}<\omega_{i',1}$ and for $i_0<i<j$ we have that $\omega_{i',1}\leq\omega_{j',1}$ and that $i'\neq j'$.  It follows that we have $\tilde\omega_{i,n}<\omega_{i,1}$ for $i>i_0$. \end{proof}

\begin{thm}\label{thm:sln}
Let $(V,G)=(\ell\C^n,\SL_n)$. Then 
$$
p_\infty^*\cO(Z_\infty)=\cO(V_\infty)^{G_\infty}.
$$
\end{thm}

\begin{proof}
Let $h\in\cO(V_\infty)^{G_\infty}$. We may assume that $h$ is homogeneous of a fixed weight. By Corollary \ref{cor:main}, $L(h)=L(P(W))$ where in the $i$th row of $W$ we have $\omega_{i,1}>\tilde \omega_{i,n}$ if $\wt(\omega_{i,n})>0$. Whenever we have $\omega_{i,1}>\tilde\omega_{i,n}$ we have that $\omega_{i,1}<\dots<\omega_{i,n}$ is of the form $(j_1,k_1)<\dots<(j_n,k_n)$ where $k_1=k_n$ or $k_n=k_1+1$ and $j_n<j_1$. In the former case we have a sequence $(j_1,a)<\dots<(j_n,a)$ where $a>0$ and $j_1<\dots<j_n$. In the latter case we have a sequence $(j_{b+1},a)<\dots<(j_n,a)<(j_1,a+1)<\dots<(j_b,a+1)$ where $j_1<\dots<j_n$ and $a\geq 0$. Thus the factor of $L(h)$ corresponding to the $i$th row of $W$ is $L(D^k([x^{(j_1,0)},\dots,x^{(j_n,0)}]))$ where $k$ is the sum of the $k_j$. Hence there is a homogeneous element $f$ of $p_\infty^*\cO(Z_\infty)$ of the same weight and degree as $h$ such that $L(h-p_\infty^*(f))>L(h)$. By induction we see that $h\in p_\infty^*\cO(Z_\infty)$.
\end{proof}
 
We have not shown that $p_\infty^*$ is injective, hence we have not shown that $\cO(Z_\infty)$ is reduced. Using techniques as in \S \ref{sec:cstar} one can show that $p_\infty^*$ is an isomorphism when $n=2$. We now show that this is not the case for $n=3$.
 
 \begin{example}\label{ex:sl3}
 Let $(V,G)=(6\C^3,\SL_3)$. Let $[abc]$ denote the determinant $[x^{(a,0)},x^{(b,0)},x^{(c,0)}]$ where $1\leq a<b<c\leq 6$. Consider the corresponding element $f_{abc}\in\cO(Z)$. Now the generators and relations transform by representations of $\SL_6$.  The generators transform by $\bigwedge^3(\C^6)$ and the relations  are generated by quadratic expressions which transform by the adjoint representation of $\SL_6$. The relations involving all six indices are those fixed by the maximal torus $T$ of $\SL_6$ and these span a space of dimension $\dim T=5$. Now consider all the possible terms $f_{abc}Df_{def}$ and $(Df_{abc})f_{def}$ where $\{a,b,c,d,e,f\}=\{1,\dots,6\}$. It is easy to see that there are six cases where the pullback of these terms to $\cO(V_1)$ contains a nonstandard monomial. Since we only have 5 relations to straighten with, we see that one of the straightening relations in $\cO(V_\infty)$ does not come from the relations of $\cO(Z)$. Hence $p_\infty^*$ is not injective. Since $p_\infty$ is dominant,   the kernel of $p_\infty^*$ consists of nilpotent elements. Hence $Z_\infty$ is not reduced. \end{example}

\end{document}